\documentclass[10pt,a4paper]{amsart}
\usepackage[english]{babel}
\usepackage{amssymb,xypic,amscd}
\CompileMatrices

\newtheorem{defn}{Definition}
\newtheorem{thm}[defn]{Theorem}

\newtheorem{cor}[defn]{Corollary}
\newtheorem{lem}[defn]{Lemma}
\newtheorem{prop}[defn]{Proposition}

\theoremstyle{remark}
\newtheorem{rem}[defn]{Remark}
\theoremstyle{remark}
\newtheorem{exam}{Example}

\numberwithin{equation}{section} \numberwithin{defn}{section}

\newcommand\aut{\operatorname{Aut}}
\newcommand\ed{\operatorname{End}}

\newcommand\Ker{\operatorname{Ker}}
\renewcommand\dim{\operatorname{dim}}

\newcommand\Det{\operatorname{det}}

\newcommand\tr{\operatorname{tr}}

\newcommand\limpl[1]{\underset{#1}\varprojlim\,}

\def\mod #1/#2{\kern.06em{\raise1.2pt\hbox{$#1$}}/
      {\raise-1.2pt\hbox{$#2$}}}

\newcommand\B{{\mathcal B}}

\renewcommand\tilde{\widetilde}

\renewcommand\lim{\limpl{A\in\B}}
\newcommand\beq{
      \setcounter{equation}{\value{defn}}\addtocounter{defn}1
      \begin{equation}}

\begin{document}

\title{Classification of Finite Potent Endomorphisms}
\author{Fernando Pablos Romo }
\address{Departamento de
Matem\'aticas and Instituto Universitario de F\'{\i}sica Fundamental y Matem\'aticas, Universidad de Salamanca, Plaza de la Merced 1-4,
37008 Salamanca, Espa\~na} \email{fpablos@usal.es}
\thanks{
\\ This work is partially supported by the
Spanish Government research contract no. MTM2012-32342.}

\maketitle

\begin{abstract} The aim of this work is to offer a family of invariants that allows us to classify finite potent endomorphisms on arbitrary vector spaces, generalizing the classification of endomorphisms on finite-dimensional vector spaces. As a particular case we classify nilpotent endomorphisms on infinite-dimensional vector spaces.
\end{abstract}

\bigskip

\setcounter{tocdepth}1

\tableofcontents
\bigskip

\section{Introduction}

The classification of mathematical objects is a classical problem: try to determine the structure of a quotient set up to some equivalence. The classification of endomorphisms on finite-dimensional vector spaces, the group of automorphisms acting by conjugation, from Jordan bases is well-known.

In this work we generalize this classification to finite potent endomorphisms on arbitrary vector spaces (this notion was introduced by J. Tate as a tool for his elegant definition of Abstract Residues -\cite{Ta}-): that is, if $\varphi$ and $\phi$ are finite potent endomorphisms on a $k$-vector space $V$, we give conditions for the existence of an automorphism $\tau \in \aut_k(V)$ such that $\phi = \tau \varphi \tau^{-1}$. As a particular case, we classify nilpotent endomorphisms on arbitrary infinite-dimensional vector spaces. As far as we know, this result is not stated explicitly in the literature.

If $X^{fp}_V$ is the subset of $\ed_k (V)$ consisting of all the finite potent endomorphisms of $V$, we also provide a explicit description of the quotient set $X^{fp}_V \big / \aut_k(V)$ for countable dimensional vector spaces.

The paper is organized as follows. In section \ref{s:pre} we briefly recall the basic definitions of this work: the definition of finite potent endomorphisms with the decomposition of the vector space given by M. Argerami, F. Szechtman and R. Tifenbach in \cite{AST}; Tate's definition of the trace of a finite potent endomorphism -\cite{Ta}-; and the definition of the determinant for these objects recently offered by D. Hern\'andez Serrano and the author in \cite{HP}. Moreover, in this section we describe the well-known theory of the classification of endomorphisms on finite-dimensional vector spaces.

Section \ref{s:main-classif} is devoted to giving the main results of this work. Indeed, we offer invariants to classify nilpotent endomorphism on arbitrary vector spaces (Theorem \ref{th:class-nilpot}). As an example we offer the explicit description of the quotient set obtained from the classification of nilpotent endomorphisms on a countable dimensional vector space (Example \ref{ex:countable}). Using this classification of nilpotent endomorphisms, we offer invariants to solve the proposed problem: the classification of finite potent endomorphisms on arbitrary vector spaces (Theorem \ref{th:class-fp}). Finally, we characterize the quotient set of finite potent endomorphisms for countable dimensional vector spaces (Example \ref{exam:fp}), and we show that the determinant and the trace of a finite potent operator are invariant under the classification offered.

\section{Preliminaries} \label{s:pre}

This section is added for the sake of completeness.

\subsection{Basic Definitions}\label{ss:CT}

 Let $k$ be an arbitrary field, and let $V$ be a $k$-vector space.

    Let us now consider an endomorphism $\varphi$ of $V$. We say
that $\varphi$ is ``finite potent'' if $\varphi^n V$ is finite
dimensional for some $n$. This definition was introduced by J. Tate in \cite{Ta} as a basic tool for his elegant definition
of Abstract Residues.

 In 2007 M. Argerami, F. Szechtman and R. Tifenbach showed in \cite{AST} that an endomorphism $\varphi$ is
finite potent if and only if $V$ admits a $\varphi$-invariant
decomposition $V = U_\varphi \oplus W_\varphi$ such that
$\varphi_{\vert_{U_\varphi}}$ is nilpotent, $W_\varphi$ is finite
dimensional, and $\varphi_{\vert_{W_\varphi}} \colon W_\varphi
\overset \sim \longrightarrow W_\varphi$ is an isomorphism.

Indeed, if $k[x]$ is the algebra of polynomials in the variable x with coefficients in $k$, we may view $V$ as an $k[x]$-module via $\varphi$, and the explicit definition of the above $\varphi$-invariant subspaces of $V$ is:
\begin{itemize}

\item $U_\varphi = \{v \in V \text{ such that } x^m v = 0 \text{ for some m }\}$.

\item $W_\varphi = \{v \in V \text{ such that } p(x) v = 0 \text{ for some } p(x) \in k[x] \text{ relative prime to } x\}$.

\end{itemize}

Note that if the annihilator polynomial of $\varphi$ is $x^m\cdot p(x)$ with $(x,p(x)) = 1$, then $U_\varphi = \Ker \varphi^m$ and $W_\varphi = \Ker p(\varphi)$.

Hence, this decomposition is unique. In this paper we shall call this decomposition the
$\varphi$-invariant AST-decomposition of $V$.

For a finite potent endomorphism $\varphi$, a trace $\tr_V(\varphi) \in k$ may
be defined as $\tr_V(\varphi) = \tr_W(\varphi_{\vert_{W_\varphi}})$

This trace has the following properties:
\begin{enumerate}
\item If $V$ is finite dimensional, then $\tr_V(\varphi)$ is the
ordinary trace. \item If $W$ is a subspace of $V$ such that
$\varphi W \subset W$ then $$\tr_V(\varphi) = \tr_W(\varphi) +
\tr_{V/W}(\varphi)\, .$$ \item If $\varphi$ is nilpotent, then
$\tr_V(\varphi) = 0$.
\end{enumerate}

For details readers are referred to \cite{Ta}.

With the previous notation, and using this AST-decomposition of $V$, recently D. Hern\'andez Serrano and the author have offered in \cite{HP} a definition
of a determinant for finite potent endomorphisms as follows:
$$\Det^k_V(1 +\varphi) := \Det^k_{W_\varphi}(1 + \varphi_{\vert_{W_\varphi}})\,.$$

This definition generalizes the one given by A. Grothendieck for operators of finite rank in \cite{Gr}, and satisfies the following properties:
\begin{itemize}
\item If $V$ is finite dimensional, then $\Det^k_V(1 + \varphi)$
is the ordinary determinant.

\item If $W$ is a subspace of $V$ such that $\varphi W \subset W$,
then: $$\Det^k_V(1 + \varphi) = \Det^k_W(1 + \varphi) \cdot
\Det^k_{V/W}(1 + \varphi)\, .$$

\item If $\varphi$ is nilpotent, then $\Det^k_V(1 + \varphi) = 1$.
\end{itemize}

\subsection{Classification of Endomorphisms on Finite-Dimensional Vector Spa\-ces}\label{ss:clas-finite}
\label{ss:ACK}

Let $E$ be a finite-dimensional vector space over a field $k$, and let $T\in \ed_k(E)$ be an endomorphism of $E$. We have that $T$ induces a structure of $k[x]$-module from the action $$\begin{aligned} k[x] \times E &\longrightarrow E \\ [p(x), e] &\longmapsto p(T)e \, .\end{aligned}$$\noindent We shall write $E_T$ to denote the vector space $E$ with this $k[x]$-module structure.

It is known that two endomorphisms $T, {\tilde T} \in \ed_k(E)$ are equivalent, i. e. there exists an automorphism $\tau \in \aut_k (V)$ such that $T = \tau {\tilde T} \tau^{-1}$ if and only if the $k[x]$-modules $E_T$ and $E_{{\tilde T}}$ are isomorphic.

Let $a_T(x) = p_1(x)^{n_1}\cdot \dots \cdot p_r(x)^{n_r}$ be the annihilator polynomial of $T$, where $p_i(x)$ are irreducible polynomials on $k[x]$. The decomposition of $k$-vector spaces $$E = \Ker p_1(T)^{n_1}\oplus \dots \oplus \Ker p_r(T)^{n_r}\, ,$$\noindent where the subspaces $\Ker p_i(T)^{n_i} \subset E$ are invariant by $T$, is compatible with the respective $k[x]$-module structures.

Indeed, the classification of endomorphisms on finite-dimensional vector spaces is reduced to studying the $k[x]$-module structure of $\Ker p(T)^n$, $p(x)$ being an irreducible polynomial on $k[x]$. This structure is determined by a decomposition of $k[x]$-modules:

$$\Ker p(T)^n \simeq {\big [ k[x] \big / p(x)^n \big ]}^{\nu_n (E,p(T))} \oplus \dots \oplus {\big [ k[x] \big / p(x) \big ]}^{\nu_1 (E,p(T))}\, ,$$\noindent where $\nu_n (E,p(T)) \ne 0$ and $$\nu_i (E,p(T)) = \dim_K \big ( \Ker p(T)^i \big / [\Ker p(T)^{i-1} + p(T) \Ker p(T)^{i+1}] \big )\, ,$$\noindent with $K = k[x] \big / p(x)$.

Again writing the annihilator polynomial of $T$ as $a_T(x) = p_1(x)^{n_1}\cdot \dots \cdot p_r(x)^{n_r}$, the invariant factors $\{\nu_i (E,p_j(T))\}_{1\leq j \leq r\, ; \, 1\leq i \leq n_j}$ determine the $k[x]$-module structure of $E_T$ and, therefore, they classify the endomorphism $T$.

Furthermore, if $K_j = k[x] \big / p_j(x)$, then $p_j(x)$ is a polynomial of degree $d_j = \dim_k K_j$.

The well-known theory of classification of endomorphisms on finite-dimensional vector spaces shows that there exist families of vectors $\{e^{ij}_h\}_{1\leq h \leq \nu_i (E,p_j(T))}$ with
 $$\begin{aligned} e^{ij}_h &\in \Ker p_j(T)^i \\ e^{ij}_h &\notin \Ker p_j(T)^{i-1} + p_j(T) \Ker p_j(T)^{i+1}\, ,\end{aligned}$$\noindent for all $1\leq j \leq r$ and $1\leq i \leq n_j$, such that if we set
 $$<e^{ij}_h>_T = \underset {0\leq s \leq i - 1} \bigoplus  <p_j(T)^{s}[e^{ij}_h], p_j(T)^{s} [T(e^{ij}_h)], \dots , p_j(T)^{s}[T^{d_j -1}(e^{ij}_h)]>\, ,$$
$$E = \underset {\begin{aligned}1 &\leq j \leq r  \\ 1 &\leq i \leq n_j \\ 1 \leq h &\leq \nu_i (E,p_j(T)) \end{aligned}} {\bigoplus} <e^{ij}_h>_T\, .$$

Indeed, the family of vectors \begin{equation} \label{eq:bases-j-finite} \underset {\begin{aligned}1 &\leq j \leq r  \\ 1 &\leq i \leq n_j  \end{aligned}} {\bigcup} \{e^{ij}_h\}_{1\leq h \leq \nu_i (E,p_j(T))}\end{equation} generates a Jordan basis of $E$ for $T$.

\section{Classification of finite potent endomorphisms}\label{s:main-classif}

Let $V$ be an arbitrary $k$-vector space, and let $\ed_k (V)$ be the $k$-vector space of endomorphisms of $V$.
Let us consider the subset $X^{fp}_V \subset \ed_k (V)$ consisting of all the finite potent endomorphisms of $V$ (note that $X^{fp}_V$ is not
a vector subspace of $\ed_k (V)$ because, in general, the sum of two finite potent endomorphisms is not finite potent).

We have an action of the group of automorphisms of $V$, $\aut_k (V)$, on $X^{fp}_V$ by conjugation:
$$\begin{aligned} G \times X^{fp}_V &\longrightarrow X^{fp}_V \\ (\tau, \varphi) &\longmapsto \tau \varphi {\tau}^{-1} \, .\end{aligned}$$

This section is devoted to offering the main result of this work: the characterization of the quotient set ${X^{fp}_V}\big /{\aut_k(V)}$.

Henceforth, if $f\in \ed_k (V)$ and $H$ is a $k$-subspace of $V$ invariant by $f$, to simplify we shall again write $f\colon H \longrightarrow H$ and $f\colon V/H \longrightarrow V/H$ to refer to the induced linear operators.

\subsection{Classification of nilpotent endomorphisms} \label{ss:nilpotent}

Let $V$ again be an arbitrary vector space over a ground field $k$, and let ${\mathcal B} =\{v_i\}_{i\in I}$ be a basis of $V$. It is known that $\dim (V) = \# {\mathcal B}$ is independent of the basis chosen, $\# {\mathcal B}$ being the cardinal of the set ${\mathcal B}$.

Let $X^{N}_V$ be the subset of $X^{fp}_V$ consisting of all nilpotent endomorphisms of $V$, that is:
$$X^{N}_V = \{f \in \ed_k (V) \text{ such that } f^n = 0 \text{ for a certain } n\in {\mathbb N}\, \}\, .$$

It is clear that the group of automorphisms of $V$, $\aut_k(V)$, acts on $X^{N}_V$ by conjugation. We shall characterize the quotient set $X^{N}_V\big /{\aut_k(V)}$.

To start, we shall construct a Jordan Basis of $V$ for a nilpotent endomorphism.

Let us consider $f \in X^{N}_V$ of order $n$ ($f^n = 0$ and $f^{n-1} \ne 0$). We have a sequence of $k$-subspaces of $V$:
$$\{0\} \subsetneqq \Ker f \subsetneqq \Ker f^2 \subsetneqq \cdots \subsetneqq  \Ker f^{n-1} \subsetneqq \Ker f^n = V\, .$$

Let $H_n^f$ be a supplementary subspace of $\Ker f^{n-1}$ on $V$, i. e. $$\Ker f^{n-1} \oplus H_n^f = V\, .$$

Note that the dimension of $H_n^f$, $\dim H_n^f$, is independent of the choice made because $$H_n^f \simeq V/{\Ker f^{n-1}}\, .$$

\begin{lem} \label{lem:barrosa} With the previous notation, we have that:
\begin{enumerate}
\item $f(H_n^f) \cap H_n^f = \{0\} \subset V$.
\item $f(H_n^f) \cap \Ker f^{n-2} = \{0\} \subset \Ker f^{n-1}$.
\end{enumerate}
\end{lem}

\begin{proof}

\begin{enumerate}
\item Since $f(H_n^f) \subset \Ker f^{n-1}$ and $H_n^f \cap \Ker f^{n-1} = \{0\}$, the first assertion is deduced.

\item Let us consider $v \in f(H_n^f) \cap \Ker f^{n-2}$.

If $v = f(h)$, with $h \in H_n^f$, since $v\in \Ker f^{n-2}$, then $f^{n-1} (h) = 0$.

Thus, $h \in \Ker f^{n-1} \cap H_n^f = \{0\}$, and we conclude that $v = 0$.
\end{enumerate}
\end{proof}

Hence, we can now consider a $k$-subspace of $V$, $H_{n-1}^f$, such that $$\Ker f^{n-1} = \Ker f^{n-2} \oplus f(H_n^f) \oplus H_{n-1}^f\, .$$

Since $V = \Ker f^n$ and $f(\Ker f^{n-1}) \subseteq \Ker f^{n-2}$, it is clear that
$$\Ker f^{n-1} = [\Ker f^{n-2} + f(\Ker f^n)] \oplus H_{n-1}^f\, .$$

Similarly to Lemma \ref{lem:barrosa}, we can prove that:
\begin{lem} If $1\leq i < n$, one has that:
$$f^{j-i} (H_j^f) \cap [f^{j-i+1} (H_{j+1}^f) + \dots + f^{n-i} (H_n^f)+ \Ker f^{i-1}] = \{0\} \in \Ker f^i\, ,$$\noindent for all $i + 1 \leq j < n$.
\end{lem}

Then, recurrently, we can fix $k$-vector subspaces of $V$, $\{H_{r}^f\}_{1\leq r \leq n}$, such that:
\begin{equation} \label{eq:adad} \Ker f^i = \Ker f^{i-1} \oplus f^{n-i} (H_n^f) \oplus \dots \oplus f(H_{i+1}^f) \oplus H_i^f\, ,\end{equation}\noindent for every $1\leq i < n$.

As above, bearing in mind that $$\Ker f^{i-1} \oplus f^{n-i} (H_n^f) \oplus \dots \oplus f(H_{i+1}^f) = \Ker f^{i-1} + f(\Ker f^{i+1})\, ,$$\noindent one has that $$\Ker f^i = [\Ker f^{i-1} + f(\Ker f^{i+1})] \oplus H_i^f\, .$$

    If we now set $$\begin{aligned}   W_n^{\varphi} &= {V}\big / {\Ker f^{n-1}} \\ W_{n-1}^{\varphi} &= \Ker f^{n-1}  \big / [\Ker f^{n-2} + f(\Ker f^{n})]
 \\ &\vdots \\ W_{i}^{\varphi} &= \Ker f^i  \big / [\Ker f^{i-1} + f(\Ker f^{i+1})] \\ &\vdots  \\ W_{2}^{\varphi} &= \Ker f^2  \big / [\Ker f + f(\Ker f^{3})] \\ W_{1}^{\varphi} &= \Ker f  \big /  f(\Ker f^2) \, ,\end{aligned}$$\noindent by construction the dimension of the $k$-subspace $H_i^f$ is independent of the choices made for all $1\leq i \leq n$, and $$\dim H_i^f = \dim W_i^{\varphi}\, .$$

Accordingly, we can assign to each $f \in X^{N}_V$ a family of cardinals $\{\mu_i(V,f)\}_{1\leq i \leq n}$ where $\mu_i(V,f) = \dim W_i^{\varphi}$.

\begin{rem} If $f \in X^{N}_V$, with $f^n = 0$ and $f^{n-1} \ne 0$, one has that $\mu_n(V,f) \ne 0$.

We should emphasize that there is no relationship of order between the invariants $\{\mu_i(V,f)\}$.

    Thus, if $V$ is a $k$-vector space of countable dimension with a basis $\{e_1,e_2,\dots,$ $e_n,\dots\}$, and we consider $f,g \in \in X^{N}_V$ defined by:
$$f(v_i) = \left \{ {\begin{aligned} v_3  &\text{ if } i = 1,2 \\ 0  &\text{ if } i \geq 3 \end{aligned}} \right . \, ,$$

$$g(v_j) = \left \{ {\begin{aligned} 0 &\text{ if } j = 1,2 \\ v_{j+1} &\text{ if } j \geq 3 \text{ odd } \\ 0 &\text{ if } j \geq 4 \text{ even } \end{aligned}} \right . \, ,$$\noindent then
$f^2 = g^2 = 0$, and one has that:
\begin{itemize}
\item  $H_2^f = <e_1>$, $f(H_2^f) = <e_3>$, and $H_1^f = <e_1 - e_2, e_4, e_5, ...>$;
\item $\mu_1 (V,f) ) = \aleph_0$, and $\mu_2 (V,f) ) = 1$;
\item $H_2^g = <e_{2j + 1}>_{j\geq 1}$, $f(H_2^g) = <e_{2j}>_{j\geq 2}$, and $H_1^g = <e_1,e_2>$;
\item $\mu_1 (V,2) ) = 2$, and $\mu_2 (V,f) ) = \aleph_0$;
 \end{itemize}
$\aleph_0$ being the cardinal of the set of all natural numbers.

\end{rem}

Henceforth, for indexing bases, $S_{{\mu}_i (V,f)}$ will be a set such that $\# S_{{\mu}_i (V,f)} = {\mu}_i (V,f)$, with $S_{{\mu}_i (V,f)} \cap S_{{\mu}_j (V,f)} = \emptyset$ for $i\ne j$. In particular, if ${\mu}_i (V,f)$ is a natural number $N$, then $S_{{\mu}_i (V,f)} = \{i_1,i_2, \dots,i_N\}$. Note that $$\# \big [ \underset {1\leq i \leq n} \bigcup (S_{{\mu}_i (V,f)} \sqcup \overset {i)} \dots \sqcup S_{{\mu}_i (V,f)}) \big ] = \dim (V)\, .$$

\begin{prop} \label{lem:barrosa} If $1\leq i \leq n$ and $\{v_{s_i}\}_{s_i\in S_{{\mu}_i (V,f)}}$ is a basis of $H_i^f$, then:
\begin{enumerate}
\item $\{f^r(v_{s_i})\}_{s_i\in S_{{\mu}_i (V,f)}}$ is a basis of $f^r(H_i^f)$ for all $1\leq r < i$.
\item $$\underset {\begin{aligned} s_i &\in S_{{\mu}_i (V,f)} \\ 1 &\leq i \leq n \end{aligned}} {\bigcup} \{v_{s_i}, \varphi (v_{s_i}), \dots , \varphi^{i-1} (v_{s_i})\}$$\noindent is a Jordan basis of $V$ for $\varphi$.
\end{enumerate}
\end{prop}
\begin{proof}
\begin{enumerate}

\item We only have to show that $\{f^r(v_{s_i})\}_{s_i\in S_{{\mu}_i (V,f)}}$ is a system of linearity-independent vectors.

If $J\subset S_{{\mu}_i (V,f)}$, and $\sum_{j\in J} \lambda_j f^r(e_j) = 0$, since $f^r$ is linear one has that $$f^r[\sum_{j\in J} \lambda_j (e_j)] = 0\, .$$

Thus, $\sum_{j\in J} \lambda_j (e_j) \in \Ker f^r \cap H_i^f = \{0\}$ with $r < i$, and we conclude that $\lambda_j = 0$ for every $j\in J$.

\item The existence of bases of $V$ with the required structure is a direct consequence of the decomposition of $V$ obtained recurrently from expressions (\ref{eq:adad}).
\end{enumerate}
\end{proof}

\begin{rem}
Recall from \cite{LB} that for every  $\phi \in \ed (V)$
possessing an annihilating polynomial of an arbitrary
infinite-dimensional vector space $V$ there exists a Jordan basis
of $V$ associated with $\phi$. We should note that the above construction of Jordan bases for nilpotent endomorphisms is compatible with the results of \cite{LB}. However, from the proof of the existence of Jordan bases given in \cite{LB} a Classification Theorem for these endomorphisms is not obtained, because from the statements of this paper it is not possible to deduce that the dimensions of the vector subspaces that determine a Jordan basis are independent of the choices made.
\end{rem}

We shall use the existence of Jordan bases for nilpotent elements to characterize the quotient set $X^{N}_V\big /{\aut_k(V)}$.

Note that a Jordan basis of $V$ for a nilpotent endomorphism of order $n$, $\varphi$, is determined by a family of vectors \begin{equation} \label{eq:uds} \{v_{s_1}\}_{s_1 \in S_{{\mu}_1 (V,f)}}\cup \dots \cup \{v_{s_n}\}_{s_n \in S_{{\mu}_n (V,f)}}\, ,\end{equation} $\{v_{s_i}\}_{s_i\in S_{{\mu}_i (V,f)}}$ being a basis of $H_i^f$.

Let $\tau$ be an automorphism of $V$.

\begin{lem} \label{lem:palace} If $v\in V$, then $$f^s (v) = 0 \Longleftrightarrow {\bar f}^s (\tau (v)) = 0$$\noindent with ${\bar f} = \tau f \tau^{-1}$.
\end{lem}

\begin{proof} One has that: $${\bar f}^s (\tau (v)) \Longleftrightarrow \tau [f^s (v)] = 0 \Longleftrightarrow f^s (v) = 0\, .$$
\end{proof}

\begin{cor} \label{cor:palace} If ${\bar f} = \tau f \tau^{-1}$, then
\begin{itemize}
\item $\tau [\Ker f^r] = \Ker {\bar f}^r$ for all $r\geq 1$.
\item $\tau [\Ker f^{i-1} + f\Ker f^{i+1}] = \Ker {\bar f}^{i-1} + {\bar f}\Ker {\bar f}^{i+1}$ for all $i\geq 1$.
\end{itemize}
\end{cor}

\begin{prop} \label{prop:avila} If ${\bar f} = \tau f \tau^{-1}$, then ${\mu}_i (V,f) = {\mu}_i (V,{\bar f})$ for all $1\leq i \leq n$.
\end{prop}

\begin{proof} Bearing in mind Corollary \ref{cor:palace} one has that

\xymatrix{
0 \ar[r] & \Ker f^{i-1} + f(\Ker f^{i+1}) \ar[r] \ar[d]^{\sim}_{\tau} &  \Ker f^i \ar[d]_{\sim}^{\tau}  \\
0 \ar[r] & \Ker {\bar f}^{i-1} + f(\Ker {\bar f}^{i+1}) \ar[r] &  \Ker {\bar f}^i
}

and, hence, for all $1\leq i \leq n$, $\tau$ induces an isomorphism of $k$-vector spaces

$$\Ker f^i  \big / [\Ker f^{i-1} + f(\Ker f^{i+1})] \simeq \Ker {\bar f}^i  \big / [\Ker {\bar f}^{i-1} + f(\Ker {\bar f}^{i+1})]\, ,$$

from where the statement can be deduced.
\end{proof}

\begin{rem}  Let $f$ be a nilpotent endomorphism of order $n$, and let $\{v_{s_1}\}_{s_1 \in S_{{\mu}_1 (V,f)}}\cup \dots \cup \{v_{s_n}\}_{s_n \in S_{{\mu}_n (V,f)}}$ be a family of vector spaces determining a Jordan basis of $V$ for $f$.

If ${\bar f} = \tau f \tau^{-1}$, one has that $${\bar f}^s [\tau (v_{s_i})] = \tau [f^s (v_{s_i})]$$\noindent for all $1\leq i \leq n$ and $s\geq 1$, and we have that
 $$\{\tau (v_{s_1})\}_{s_1 \in S_{{\mu}_1 (V,f)}}\cup \dots \cup \{\tau (v_{s_n})\}_{s_n \in S_{{\mu}_n (V,f)}}$$\noindent determines a Jordan basis of $V$ for ${\bar f}$.
\end{rem}

\begin{thm}[Classification Theorem] \label{th:class-nilpot} Let $f, g \in X^{N}_V$ be two nilpotent endomorphisms of order $n$. Thus, $f \sim g$ (mod. $\aut_k(V)$) if and only if ${\mu}_i (V,f) = {\mu}_i (V,g)$ for all $1\leq i \leq n$.
\end{thm}

\begin{proof}

$\Longrightarrow )$ If $[f] = [g] \in X^{N}_V\big /{\aut_k(V)}$, there exists $\tau \in \aut_k (V)$ such that $g = \tau f \tau^{-1}$, and it follows from Proposition \ref{prop:avila} that ${\mu}_i (V,f) = {\mu}_i (V,g)$ for all $1\leq i \leq n$.

$\Longleftarrow )$ If ${\mu}_i (V,f) = {\mu}_i (V,g)$ for all $1\leq i \leq n$, let us consider two families of vectors $\{v_{s_1}\}_{s_1 \in S_{{\mu}_1 (V,f)}}\cup \dots \cup \{v_{s_n}\}_{s_n \in S_{{\mu}_n (V,f)}}$ and $\{w_{s_1}\}_{s_1 \in S_{{\mu}_1 (V,g)}}\cup \dots \cup \{w_{s_n}\}_{s_n \in S_{{\mu}_n (V,g)}}$ determining Jordan bases of $V$ for $f$ and $g$ respectively.

Let $\tau \in \aut_k(V)$ be the automorphism defined by $$\tau [f^r (v_{s_i})] = g^r (w_{s_i})\, ,$$\noindent for all $s_i \in S_{{\mu}_i (V,f)}$, $1\leq i \leq n$ and $0\leq r < i$.

By construction, one has that $$(\tau f \tau^{-1})[g^r (w_{s_i})] = (\tau f) [f^r (v_{s_i})] = \tau [f^{r+1} (v_{s_i})] = g^{r+1} (w_{s_i})\, .$$

Accordingly, $\tau f \tau^{-1} = g$ and $[f] = [g] \in X^{N}_V\big /{\aut_k(V)}$.
\end{proof}

\begin{exam} \label{ex:countable} Let $V$ be a $k$-vector space of countable dimension. If $Y = \{0\} \cup {\mathbb N} \cup \{\aleph_0\}$, $\aleph_0$ being the cardinal of the set of all natural numbers, one has that
$$X^{N}_V\big /{\aut_k(V)} = \underset {n\in {\mathbb N}} \bigcup [ {\prod_n}' Y]\, ,$$\noindent where $${\prod_n}' Y = \{(y_1, \dots , y_n) \text{ with } y_i\in Y\, , \, y_n \ne 0 \, , \text{ and } y_j = \aleph_0 \text{ for at least one j}\}\, .$$
\end{exam}

\begin{rem} If $E$ is a finite-dimensional $k$ vector space, and $f\in \ed_k (E)$ is nilpotent of order $n$ with invariants $\{{\mu}_i (V,f)\}_{1\leq i \leq n}$, we should note that ${\mu}_i (V,f) = {\nu}_i (V,f)$ for all $i$ (see Subsection \ref{ss:clas-finite}), and, hence, the structure of $E$ as a $k[x]$-module induced by $f$ is: $$E_f \simeq {\big ( k[x]/x \big )}^{{\mu}_1 (V,f)} \oplus \dots \oplus {\big ( k[x]/x^i \big )}^{{\mu}_i (V,f)}\oplus \dots \oplus {\big ( k[x]/x^n \big )}^{{\mu}_n (V,f)} \ \, .$$
\end{rem}

\subsection{Invariants for finite potent endomorphisms}\label{s:finite-potent}

Let $V$ be an arbitrary $k$-vector space, and let $\varphi$ be a finite potent endomorphism of $V$. Let us consider the AST-decomposition of $V$ induced by $\varphi$, that is: $V = U_\varphi \oplus W_\varphi$ such that
$\varphi_{\vert_{U_\varphi}}$ is nilpotent, $W_\varphi$ is finite
dimensional, and $\varphi_{\vert_{W_\varphi}} \colon W_\varphi
\overset \sim \longrightarrow W_\varphi$ is an isomorphism.

Let $\tau \in \aut_k (V)$ again be an automorphism of $V$.

\begin{lem} If ${\bar \varphi} = \tau \varphi \tau^{-1}$, and $V = U_{\bar \varphi} \oplus W_{\bar \varphi}$ is the AST-decomposition of $V$ induced by $\bar \varphi$, then $\tau [U_\varphi] = U_{\bar \varphi}$ and $\tau [W_\varphi] = W_{\bar \varphi}$.
\end{lem}

\begin{proof} Since, $$U_\varphi = \{v \in V \text{ such that } x^m v = 0 \text{ for some m }\}$$\noindent and
$$W_\varphi = \{v \in V \text{ such that } f(x) v = 0 \text{ for some } f(x) \in k[x] \text{ relative prime to } x\}\, ,$$\noindent bearing in mind that $$p({\bar \varphi}) (\tau v) = 0 \Longleftrightarrow   [\tau p(\varphi) \tau^{-1}] (\tau v) = 0 \Longleftrightarrow \tau [p(\varphi) ( v)] = 0 \Longleftrightarrow  [p(\varphi) ( v)] = 0 $$\noindent for all $v\in V$, the claim is deduced.
\end{proof}

\begin{cor} If ${\bar \varphi} = \tau \varphi \tau^{-1}$, then:
\begin{itemize}
\item $\tau [\varphi_{\vert_{U_\varphi}}] \tau^{-1} = {\bar \varphi}_{\vert_{U_{\bar \varphi}}}$.
\item $\tau [\varphi_{\vert_{W_\varphi}}] \tau^{-1} = {\bar \varphi}_{\vert_{W_{\bar \varphi}}}$.
\end{itemize}
\end{cor}

With similar arguments to Lemma \ref{lem:palace} and Corollary  \ref{cor:palace} one has that
\begin{lem} \label{lem:tata} If ${\bar \varphi} = \tau \varphi \tau^{-1}$, then
\begin{itemize}
\item $\tau [\Ker \, (\varphi_{\vert_{U_\varphi}})^r] = \Ker \,  ({\bar \varphi}_{\vert_{U_{\bar\varphi}}})^r \subset U_{\bar \varphi}$ for all $r\geq 1$.
\item $\tau [({\varphi}_{\vert_{U_{\varphi}}}) \Ker  \, (\varphi_{\vert_{U_\varphi}})^{i}] = ({\bar \varphi}_{\vert_{U_{\bar\varphi}}}) \Ker  \, ({\bar \varphi}_{\vert_{U_{\bar\varphi}}})^{i} \subset U_{\bar \varphi}$ for all $i\geq 1$.
\end{itemize}
\end{lem}

And, analogously, one can see that

\begin{lem} \label{lem:nano} If ${\bar \varphi} = \tau \varphi \tau^{-1}$, and $p(x)\in k[x]$, then
\begin{itemize}
\item $\tau [\Ker \,  p(\varphi_{\vert_{W_\varphi}})^r] = \Ker  \, p({\bar \varphi}_{\vert_{W_{\bar\varphi}}})^r \subset W_{\bar \varphi}$ for all $r\geq 1$.
\item $\tau [(\varphi_{\vert_{W_\varphi}}) \Ker  \, p(\varphi_{\vert_{W_\varphi}})^{i}] =  p({\bar \varphi}_{\vert_{W_{\bar\varphi}}})\Ker  \, ({\bar \varphi}_{\vert_{W_{\bar\varphi}}})^{i} \subset W_{\bar \varphi}$ for all $i\geq 1$.
\end{itemize}
\end{lem}

If $\varphi_{\vert_{U_\varphi}}$ is a nilpotent endomorphism of order n of $V_{\varphi}$, and $$a_{\varphi_{\vert_{W_\varphi}}}(x) = p_1(x)^{n_1}\cdot \dots \cdot p_r(x)^{n_r}$$\noindent is the annihilator polynomial of $\varphi_{\vert_{W_\varphi}}$, let us consider the invariant factors that classify $\varphi_{\vert_{U_\varphi}} \in X^N_{U_\varphi}$ and $\varphi_{\vert_{W_\varphi}} \in \ed_k (W_\varphi)$:
\begin{itemize}
\item $\{{\mu}_i (U_\varphi,\varphi_{\vert_{U_\varphi}})\}_{1\leq i \leq n}$ (see Subsection \ref{ss:nilpotent});
\item $\{\nu_s (W_\varphi,p_j(\varphi_{\vert_{W_\varphi}}))\}_{1\leq j \leq r\, ; 1 \leq s \leq n_j}$ (see Subsection \ref{ss:clas-finite}).
\end{itemize}

It is clear that
\begin{lem} Keeping the previous notation, if $\varphi_{\vert_{U_\varphi}}$ is a nilpotent endomorphism of order n of $U_{\varphi}$, and $a_{\varphi_{\vert_{W_\varphi}}}(x)$ is the annihilator polynomial of $\varphi_{\vert_{W_\varphi}}$, then
\begin{enumerate}
\item ${\bar \varphi}_{\vert_{U_{\bar \varphi}}}$ is also a nilpotent endomorphism of order n of $U_{\bar \varphi}$\, ;
\item $a_{\varphi_{\vert_{W_\varphi}}}(x) = a_{{\bar \varphi}_{\vert_{W_{\bar \varphi}}}}(x)$.
\end{enumerate}
\end{lem}

\begin{prop} \label{prop:avila-2} If $\varphi$ is a finite potent endomorphism of $V$, and ${\bar \varphi} = \tau \varphi \tau^{-1}$, one has that:
\begin{itemize}
\item $\{{\mu}_i (U_\varphi,\varphi_{\vert_{U_\varphi}})\} = \{{\mu}_i (U_{\bar \varphi},{\bar \varphi}_{\vert_{U_{\bar \varphi}}})\}$ for all ${1\leq i \leq n}$;
\item $\{\nu_s (W_\varphi,p_j(\varphi_{\vert_{W_\varphi}}))\} = \{\nu_s (W_{\bar \varphi},p_j({\bar \varphi}_{\vert_{W_{\bar \varphi}}}))\}$ for all $1\leq j \leq r\ ; \, 1 \leq s \leq n_j$.
\end{itemize}
\end{prop}

\begin{proof} Similarly to Proposition \ref{prop:avila}, the claim is a direct consequence of the statements of Lemma \ref{lem:tata} and Lemma \ref{lem:nano} because if $\psi$ is a finite potent endomorphism with $$a_{\psi_{\vert_{W_\psi}}}(x) = p_1(x)^{n_1}\cdot \dots \cdot p_r(x)^{n_r}\, ,$$\noindent then $$\{{\mu}_i (U_\psi,\psi_{\vert_{U_\psi}})\} = \# \big ( \Ker (\psi_{\vert_{U_\psi}})^i  \big / [\Ker (\psi_{\vert_{U_\psi}})^{i-1} + \psi(\Ker (\psi_{\vert_{U_\psi}})^{i+1})]\big )$$\noindent and $\{\nu_s (W_\psi,p_j(\psi_{\vert_{W_\psi}}))\} =$ $$ \dim_{K_j} ( \Ker p_j(\psi_{\vert_{W_\psi}})^s \big / [\Ker p_j(\psi_{\vert_{W_\psi}})^{s-1} + p_j(\psi_{\vert_{W_\psi}}) \Ker p_j(\psi_{\vert_{W_\psi}})^{s+1}] \big )\, ,$$\noindent where $K_j = k[x]/p_j(x)$.
\end{proof}

\begin{thm}[Classification Theorem] \label{th:class-fp} Let $\varphi, \phi \in X^{fp}_V$ be two finite potent endomorphisms of an arbitrary $k$-vector space $V$. Thus, $\varphi \sim \phi$ (mod. $\aut_k(V)$) if and only if
\begin{itemize}
\item $\{{\mu}_i (U_\varphi,\varphi_{\vert_{U_\varphi}})\} = \{{\mu}_i (U_{\phi},{\phi}_{\vert_{U_{\phi}}})\}$ for all ${1\leq i \leq n}$;
\item $\{\nu_s (W_\varphi,p_j(\varphi_{\vert_{W_\varphi}}))\} = \{\nu_s (W_{\phi},p_j({\phi}_{\vert_{W_{\phi}}}))\}$ for all $1\leq j \leq r\ ; \, 1 \leq s \leq n_j$.
\end{itemize}
\end{thm}

\begin{proof} $\Longrightarrow )$ If $[\varphi] = [\phi] \in X^{fp}_V\big /{\aut_k(V)}$, there exists $\tau \in \aut_k (V)$ such that $\phi = \tau \varphi \tau^{-1}$, and hence $\varphi_{\vert_{U_\varphi}}$ and $\phi_{\vert_{U_\phi}}$ have the same order of nilpotency, and $a_{\varphi_{\vert_{W_\varphi}}}(x) = a_{\phi_{\vert_{W_\phi}}}(x)$.

Accordingly, it follows from Proposition \ref{prop:avila-2} that
\begin{itemize}
\item $\{{\mu}_i (U_\varphi,\varphi_{\vert_{U_\varphi}})\} = \{{\mu}_i (U_{\phi},{\phi}_{\vert_{U_{\phi}}})\}$ for all ${1\leq i \leq n}$;
\item $\{\nu_s (W_\varphi,p_j(\varphi_{\vert_{W_\varphi}}))\} = \{\nu_s (W_{\phi},p_j({\phi}_{\vert_{W_{\phi}}}))\}$ for all $1\leq j \leq r\ ; \, 1 \leq s \leq n_j$.
\end{itemize}
$\Longleftarrow )$ Let us now assume that the invariant factors of $\varphi$ and $\phi$ referred to in the statement are equal. Let us consider families of vectors: ${\bigcup}_{1\leq i \leq n} \{u_{s_i}\}_{s_i\in S_{{\mu}_i (U_\varphi,\varphi_{\vert_{U_\varphi}})}}$ determining a Jordan basis of $U_\varphi$ for $\varphi_{\vert_{U_\varphi}}$; ${\bigcup}_{1\leq i \leq n} \{v_{s_i}\}_{s_i\in S_{{\mu}_i (U_\phi,\phi_{\vert_{U_\phi}})}}$ determining a Jordan basis of $U_\phi$ for $\phi_{\vert_{U_\phi}}$; ${\bigcup}_{1 \leq j \leq r  \, ; \, 1 \leq s \leq n_j} \{w^{sj}_h\}_{1\leq h \leq \nu_s (W_{\varphi},p_j({\varphi}_{\vert_{W_{\varphi}}}))}$ genera\-ting a Jordan basis of $W_{\varphi}$ for ${\varphi}_{\vert_{W_{\varphi}}}$; and ${\bigcup}_{1 \leq j \leq r  \, ; \, 1 \leq s \leq n_j} \{e^{sj}_h\}_{1\leq h \leq \nu_s (W_{\phi},p_j({\phi}_{\vert_{W_{\phi}}}))}$ genera\-ting a Jordan basis of $W_{\phi}$ for ${\phi}_{\vert_{W_{\phi}}}$ -see (\ref{eq:bases-j-finite}) and (\ref{eq:uds})-.

Thus, we can construct isomorphisms $\tau_1\colon U_\varphi \longrightarrow U_\phi$ and $\tau_2\colon W_\varphi \longrightarrow W_\phi$ from the following assignations: $$\begin{aligned}\tau_1 [\varphi^a (u_{s_i})] &= \phi^a (v_{s_i}) \\ \tau_2 [p_j^b(\varphi) [(\varphi^c (w^{sj}_h)]] &= p_j(\phi)^b (\phi^c (e^{sj}_h)]\end{aligned}$$\noindent for all $1\leq i \leq n$; $s_i\in S_{{\mu}_i (U_\varphi,\varphi_{\vert_{U_\varphi}})}$; $0\leq a \leq i-1$; $1 \leq j \leq r$; $1 \leq s \leq n_j$; $1\leq h \leq \nu_s (W_{\phi},p_j({\phi}_{\vert_{W_{\phi}}}))$; $0\leq b \leq s-1$ and $0\leq c \leq d_j-1$.

Moreover, since $V = U_\varphi \oplus W_\varphi$ and $V = U_\phi \oplus W_\phi$, there exists a unique automorphism $\tau\in \aut_k(V)$ such that $\tau_{\vert_{U_{\varphi}}} = \tau_1$ and $\tau_{\vert_{W_{\varphi}}} = \tau_2$. Thus, an easy check shows that $\phi = \tau \varphi \tau^{-1}$, and, therefore, $[\varphi] = [\phi] \in X^{fp}_V\big /{\aut_k(V)}$.
\end{proof}

\begin{rem} If $V$ and ${\tilde V}$ are two $k$-vector spaces, and $\psi \colon V\longrightarrow {\tilde V}$ is an isomorphism, it follows from the above statements that the induced map $$\begin{aligned} \psi_*\colon X^{fp}_{V}\big /{\aut_k({V})} &\longrightarrow X^{fp}_{\tilde V}\big /{\aut_k({\tilde V})} \\ [{\varphi}] &\longmapsto [\psi {\varphi} \psi^{-1}]  \end{aligned}$$\noindent is bijective.
\end{rem}

\begin{exam} \label{exam:fp} Let $V$ be a $k$-vector space of countable dimension. For each finite potent endomorphism $\varphi \in X^{fp}_V$, since $W_{\varphi}$ is a finite-dimensional vector space, one has that $V$ and $U_\varphi$ are isomorphic as $k$-vector spaces.

For each $n\in {\mathbb N}$, let $E_n$ be a $k$-vector space with $\dim_k (E_n) = n$. If we consider the quotient sets $X_n = \ed_k(E_n)\big / \aut_k(E_n)$, Theorem \ref{th:class-fp} shows that $$X^{fp}_V\big /{\aut_k(V)} = X^{N}_V\big /{\aut_k(V)} \times \big [ \bigcup_{n\in {\mathbb N}} X_n \big ]\, ,$$\noindent $X^{N}_V\big /{\aut_k(V)}$ being the set characterized in Example \ref{ex:countable}.
\end{exam}

\subsubsection{Trace and Determinant of Finite Potent Endomorphisms} Given an arbitrary $k$-vector space $V$, for each finite potent endomorphism $\varphi$ on $V$, and from the AST-decomposition $V = U_{\varphi} \oplus W_{\varphi}$, it is possible to define a trace $\tr_V(\varphi)$ and a determinant $\Det^k_V(1 +\varphi)$ (Subsection \ref{ss:CT}).

    If $\tau \in \aut_k(V)$, it is known that:
$$ \Det^k_V(1 + \tau \varphi \tau^{-1}) = \Det^k_V(1 + \varphi) \qquad (\text{\cite{HP}}, \text{ Lemma }3.13)\, .$$

Thus, considering the map $$\begin{aligned} \delta \colon X^{fp}_V &\longrightarrow k \\ \varphi &\longmapsto \Det^k_V(1 + \varphi)\end{aligned}$$\noindent the commutative diagram

$$\xymatrix{
X^{fp}_V \ar[rr]^{\delta} \ar[d]_{\pi} & & k \\
X^{fp}_V\big /{\aut_k(V)} \ar[urr]_{\tilde \delta}& &
}$$\noindent makes sense with ${\tilde \delta} ([\varphi]) = \Det^k_V(1 + \varphi)$, $\pi$ being the quotient map.

On the other hand, if $c_{\varphi_{\vert_{W_\varphi}}} (x)$ is the characteristic polynomial of the endomorphism $\varphi_{\vert_{W_\varphi}} \in \ed_k (W_\varphi)$, it is clear that \begin{equation} \label{eq:carac} c_{\varphi_{\vert_{W_\varphi}}} (x) = c_{\phi_{\vert_{W_\phi}}} (x)\, ,\end{equation} for all endomorphisms $\varphi \sim \phi$ (mod. $\aut_k(V)$).

    Since $c_{\varphi_{\vert_{W_\varphi}}} (x) = x^m - \tr_V(\varphi)x^{m-1} + \dots$, with $m = \dim (W_\varphi)$, from (\ref{eq:carac}) we deduce that $$\tr_V(\varphi) = \tr_V (\phi)\, ,$$\noindent for all $\varphi, \phi \in X^{fp}_V$ such that $[\varphi] = [\phi] \in X^{fp}_V\big /{\aut_k(V)}$.

Then, writing $$\begin{aligned} \gamma \colon X^{fp}_V &\longrightarrow k \\ \varphi &\longmapsto \tr_V(\varphi)\, ,\end{aligned}$$\noindent one has the following commutative diagram

$$\xymatrix{
X^{fp}_V \ar[rr]^{\gamma} \ar[d]_{\pi} & & k \\
X^{fp}_V\big /{\aut_k(V)} \ar[urr]_{\tilde \gamma}& &
}$$\noindent with ${\tilde \gamma} ([\varphi]) = \tr_V(\varphi)$.

\begin{rem}[Final Consideration] We have shown that the determinant and the trace of a finite potent operator are invariant under the classification offered. It should be note that this invariance also holds for all the coefficients of the \linebreak characteristic polynomial $c_{\varphi_{\vert_{W_\varphi}}} (x)$. Thus, it makes sense to define objects from these coefficients, and to explore their properties, to attempt to obtain new intrinsic reciprocity laws using similar arguments to the proofs of the Residue Theorem -\cite{Ta}- or the Reciprocity Law for the Segal-Wilson pairing -\cite{HP}-.
\end{rem}

\end{document}